\documentclass[11pt]{ArtEnFancyClassPDC}   
\usepackage{AuxiliarPDC}

\newcommand{\muf}[1][\varphi]{\mathrm{m}_{\scriptstyle #1}}

\newcommand{\mux}[1][x]{\mu^u_{#1}}

\newcommand{\nux}[1][x]{\nu^u_{#1}}

\newcommand{\Jac}{\mathrm{Jac}}

\newcommand{\Jacu}[1][x_0,y_0]{\mathrm{Jac}^{u}_{#1}}

\newcommand{\Good}{\mathrm{Good}}
\newcommand{\Bad}{\mathrm{Bad}}

\newcommand{\Coc}[1][\Phi]{\mathscr{Coc}(#1)}
\newcommand{\Eq}{\mathscr{Eq}}

\newcommand{\h}{\mathfrak{h}}
\newcommand{\g}{\mathfrak{g}} 

\hypersetup{
	unicode=true,          
	pdftoolbar=true,        
	pdfmenubar=true,        
	pdffitwindow=false,     
	pdfstartview={FitH},    
	pdftitle={Contributions to the Ergodic Theory of Symbolic Flows},    
	pdfauthor={Pablo D. Carrasco and Federico Rodriguez-Hertz},     
	pdfkeywords={Equilibrium States, Conditional Measures, Unique Ergodicity}, 
	pdfnewwindow=true,      
	colorlinks=true,       
	linkcolor=arsenic,          
	citecolor=green,        
	filecolor=magenta,      
	urlcolor=cyan           
}

\begin{document}


	\title{Contributions to the ergodic theory of hyperbolic flows: unique ergodicity for quasi-invariant measures and equilibrium states for the time-one map}
	\author[1]{Pablo D. Carrasco\thanks{P.D.C. is partially supported by FAPEMIG Universal APQ-02160-21}}
	\author[2]{Federico Rodriguez-Hertz\thanks{F. R.-H. is partially supported by NSF grant DMS-1900778}}
    \affil[1]{ICEx-UFMG, Avda. Presidente Antonio Carlos 6627, Belo Horizonte-MG, BR31270-90}
	\affil[2]{Penn State, 227 McAllister Building, University Park, State College, PA16802}

	\date{\today}

	\maketitle

	\begin{abstract}
	We consider the horocyclic flow corresponding to a (topologically mixing) Anosov flow or diffeomorphism, and establish the uniqueness of transverse quasi-invariant measures with H\"older Jacobians. In the same setting, we give a precise characterization of the equilibrium states of the hyperbolic system, showing that existence of a family of Radon measures on the horocyclic foliation such that any probability (invariant or not) having conditionals given by this family, necessarily is the unique equilibrium state of the system.
	\end{abstract}

    \section{Introduction and Main Results} 
\label{sec:introduction_and_main_results}

In this article we study aspects related to the thermodynamic formalism for systems having some hyperbolicity, as the time-one map of an hyperbolic flow, or an Anosov diffeomorphism. This is one of the most important pieces in smooth ergodic theory, encompassing several examples and general theorems. Although this part of the theory is fairly complete, the results discussed here are not present in the literature (in this generality, as far as we know), and therefore the present project has a two-fold purpose; on the one hand, introducing the aforementioned results, and on the other serve a (crude) blueprint for future research in more general cases. 

We will rely on our previous work \cite{EqStatesCenter}, and in fact the present manuscript can be seen as a complement to this other one. Nevertheless, we will give a self-contained presentation and no a priori knowledge of this other work is necessary. 

By an hyperbolic flow we mean a ($\mathcal{C}^1$) flow $\Phi=\{\Phi_t:M\to M\}_{t\in\Real}$ of a closed manifold satisfying: there exist a Riemannian metric on $M$, a splitting of the form $TM=E^{s}\oplus E^c\oplus E^{u}$, and $0<\lambda<1$ such that the following conditions are true.
\begin{enumerate}
	\item $E^c$ is generated by the vector field tangent to the flow.
	\item $D_x\Phi_t\big(E^{s/u}(x)\big)=E^{s/u}(\Phi_t(x))$, for all $t\in\Real$, for every $x\in M$.
	\item $\sup_{x\in M}\norm{D_x\Phi_t|E^{s}(x)}, \sup_{x\in M}\norm{D_x\Phi_{-t}|E^{u}(x)}\leq \lambda^t$, for all $t\geq 0$. 
\end{enumerate}
We call $E^{s}, E^{u}$ the strong stable and strong unstable bundle, respectively, and $E^c$ will be referred as the center bundle. A hyperbolic flow is said to be of codimension-one if $\dim E^{u}=1$. The prototypical example of codimension-one hyperbolic flow is the geodesic flow corresponding to a (closed) surface of negative sectional curvature, but even in dimension three there are many other Anosov flows which are not of this type. We refer the reader to the treatise \cite{Fisher2019} for an up to date introduction and examples.

The concept of hyperbolic diffeomorphism is analogous; we say that a diffeomorphism $f:M\to M$ of an closed manifold $M$ is hyperbolic (or Anosov) if there exist a Riemannian metric on $M$, a splitting of the form $TM=E^{s}\oplus E^{u}$, and $0<\lambda<1$ such that the following conditions are true.
\begin{enumerate}
	\item $D_xf\big(E^{s/u}(x)\big)=E^{s/u}(f(x))$, for every $x\in M$.
	\item $\sup_{x\in M}\norm{D_xf|E^{s}(x)}, \sup_{x\in M}\norm{D_xf|E^{u}(x)}\leq \lambda$. 
\end{enumerate}

\begin{remark}
For the rest of the article we will work only with (codimension-one) hyperbolic flows, leaving the direct modifications for the diffeomorphism case to the reader.
\end{remark}

Now fix a codimension-one hyperbolic flow; by the classical Hadamard-Perron's stable manifold theorem, it follows in this setting that $E^u$ integrates, provided that it is orientable, to a (in general, only transversally continuous) flow $\Phi^u=\{\Phi^u_t:M\to M\}_{t\in \Real}$, called the horocyclic flow. It turns out that if we further assume transitivity of $\Phi$, meaning the existence of a dense orbit, there are only two possibilities \cite{AnosovFlow}; either
\begin{enumerate}
	\item $M$ is a fiber bundle over $S^1$, with fibers transverse to the flow lines (the suspension case), or
	\item $\Phi^u$ is minimal, that is, every orbit of $\Phi^u$ is dense in $M$. 
\end{enumerate}
The geodesic flow corresponding to a negatively curved surface falls into this last category.

\smallskip 
 
\noindent\textbf{Convention.} For the rest of this article we assume that the unstable bundle of any considered hyperbolic flow is orientable. Modulo a two-fold covering, this not a significant loss of generality.

\smallskip

We then have the following.

\begin{theorem}
If $\Phi$ is a (transitive) codimension one Anosov flow that is not a suspension, then $\Phi^u$ is uniquely ergodic. That is, there exists only one (probability) measure invariant under $\Phi^u$.
\end{theorem}

Transverse and invariant measures are in one-to-one correspondence for flows, so the above theorem says that there is only one transverse invariant measure for the flow foliation. The previous theorem is an important result originally due to H. Furstenberg \cite{Furstenberg1973} for the geodesic flow corresponding to hyperbolic surfaces, and later generalized by B. Marcus \cite{Marcus1975} to the variable curvature case, and in fact to arbitrary codimension-one Anosov flows (that are not suspension, that is). The most general result is due to R. Bowen and B. Marcus, where they do not assume one-dimensionality of the unstable bundle \cite{Bowen1977}. 

In this article we generalize the above result. Let us introduce some definitions.

\begin{definition}\label{def:transverse measure}
Let $\Psi$ be a flow on $M$. A (local) transversal to the flow is a sub-manifold $D\subset M$ that is transverse to the flow direction. A pairwise disjoint family $\mathscr{D}=\{D\}_{\scriptstyle D\ \mathrm{transversal}}$ is a global transversal if given $x\in M$ there exist $t\in\Real$ and $D\in \mathscr{D}$ such that $\Phi_t(x)\in D$. The global transversal $\mathscr{D}$ is full if for every $x\in M$ there exists $D\in\mathscr{D}$ with $x\in D$.

A transverse measure for $\Phi$ is a family of measures $\nu=\{\nu_D: D\in\mathscr{D}\}$ where $\mathscr{D}$ is a global transversal to $\Phi$ and for each $D\in \mathscr{D}$, $\nu_D$ is a Radon measure on $D$: $\mathscr{D}$ is the support of $\nu$. We will assume that the support of the transverse measures considered is full\footnote{This is not loss of generality for our purposes.} 
\end{definition}

If $\nu=\{\nu_D: D\in\mathscr{D}\}$ is a transverse measure to $\Phi$ we also write $\nu=\{\nu_x\}_{x\in M}$ where
$\nu_x=\nu_{D}, x\in D$.

\begin{definition}
Let $\Psi$ be a flow on $M$ and $\{\nu_x\}_x$ be a transverse measure for $\Psi$. We say that $\{\nu_x\}_x$ is quasi-invariant if there exist a family of positive functions $\Jac_{\nu}=\{\Jac_{x_0,y_0}:M\to \Real_{>0}:y_0=\Psi_{t_0}(x_0), x_0\in M, t_0\in\Real\}$ and $C>0$ such that 
\[
	\Psi_{-t_0}\nu_{y_0}= C\Jac_{x_0,y_0}\nu_{x_0}.
\]
In this case $\Jac_{\nu}$ is the Jacobian of the quasi-invariant measure, whereas $C$ is the associated normalization constant.
\end{definition}

\begin{remark}
Note that if $\Jac_{\nu}$ is a Jacobian, then it defines a multiplicative cocycle over the flow. 
\end{remark}

An invariant (transverse) measure corresponds to $\Jac_{\nu}\equiv 1$. We will be interested in quasi-invariant transverse measures for $\Phi^u$ having well-behaved Jacobian, and having some relation with the dynamics $\Phi$.
 As a convenient family of transversals for the horocyclic flow we consider the leaves of the center stable foliation $\Fcs=\{W^{cs}(x)\}_{x\in M}$, which are tangent to $E^{cs}=E^{c}\oplus E^s$. This choice of transversals is implicit in what follows, and in particular transverse measures to $\Phi^u$ are assumed to be supported in center stable leaves. The existence and basic properties of invariant foliations associated to $\Phi$ are recalled in section 2. 

We first state the following.

\begin{proposition}[Proposition 2.10 in \cite{EqStatesCenter}]\label{pro:quasinvariancecocycle}
Suppose that $\mu^{cs}=\{\mu^{cs}_x\}$ is a transverse measure for $\Phi^u$ such that $\forall x\in M$
\[
	\Phi_{-1}\mu^{cs}_{\Phi_1(x)}=C h\mu_{x}^{cs},
\]
where $C>0$ and $h:M\to\Real_{+}$ is H\"older. Then $\mu^{cs}$ is $\Phi^u$-quasi invariant, with Jacobian
\[
	\Jacu[x_0,y_0](x)=\prod_{j=1}^{\oo}\frac{h(\Phi_{-j}\circ \Phi^u_{t(x)} x)}{h(\Phi_{-j}(x))}.
\]
where for $x\in\Wcs{x_0}, \Phi^u_{t(x)}(x)\in\Wcs{y_0}$.
\end{proposition}

Consider the set of H\"older multiplicative cocycles
\begin{equation}\label{eq:cocicles}
 \Coc=\{h:M\to\Real_{+}: h(x)=e^{\int_0^1 \varphi(\Phi_t x)dt}\text{ for some H\"older function }\varphi\}.
 \end{equation} 
This is a natural class of functions for $\Phi$, related to its equilibrium states; see below. In the next section we recall the by now classical fact that any equilibrium state of the flow $\Phi$ corresponding to a H\"older potential $\varphi$ determines $h\in\Coc$ and a $\Phi^u$-transverse quasi invariant measure $\mu^{cs}$ satisfying the hypotheses of the previous proposition. Given $h\in\Coc$ define the family of functions 
\begin{equation}\label{eq:hjacobiano}
\mathscr{h}=\{H_{x_0,y_0}:x_0,y_0\in M, y_0\in\Wu{x_0}\}
\end{equation}
with
\begin{equation}
H_{x_0,y_0}(x)=\prod_{j=1}^{\oo}\frac{h(\Phi_{-j}\circ \Phi^u_{t(x)} x)}{h(\Phi_{-j}(x))},\quad x\in\Wcs{x_0}, \Phi^u_{t(x)}(x)\in\Wcs{y_0}.
\end{equation}

Our first main result is as follows.

\begin{theoremA}\hypertarget{theoremA}{}		
Let $\Phi$ be a codimension-one Anosov flow of class $\mathcal{C}^2$ that is not a suspension, and consider $h\in\Coc$. Assume that the unstable bundle of $\Phi$ is orientable and denote by $\Phi^u$ the induced horocyclic flow. Then there exists $\mu^{cs}=\{\mu_x^{cs}\}_x$ a transverse measure for $\Phi^u$ such that $\mu^{cs}$ is the unique quasi-invariant measure with Jacobian given by the family $\mathscr{h}$ determined by $h$.
\end{theoremA}

Given a flow $\Psi=\{\psi_t\}_t$ on $M$ and $J$ a non-negative multiplicative cocycle over $\Psi$, a (Borel) measure $\mu$ is said to be $J$-conformal if $\psi_t^{-1}\mu=J_t\cdot \mu$; the related definition for diffeomorphism is similar. The notion was introduced by Patterson \cite{Patterson1976} in the context of limit sets for Fuchsian groups, and has an important role in geometry and ergodic theory. The reader is directed to the book \cite{FeliksPrzytycki2011} for an extensive discussion of conformal measures in ergodic theory, and to the survey \cite{overviewpattersonsullivan} for a discussion of applications in geometry, in particular of the Patterson-Sullivan theory. 

We single out the following remarkable existence and uniqueness result for conformal measures due to A. Douady and J-C. Yoccoz \cite{Douady1999}.

\begin{theorem}
Let $f:\Tor\to\Tor$ be a $\mathcal{C}^2$ diffeomorphism of the circle with irrational rotation number. Then for every $s\in \Real$ there exists a unique $s\cdot Df$-conformal measure. 
\end{theorem}

The following is a re-writing of \hyperlink{theoremA}{Theorem A}, which in view of the importance of the concept of conformal measures we isolate it as a corollary.

\begin{corollaryA}
Under the same hypotheses of \hyperlink{theoremA}{Theorem A}, the horocyclic flow $\Phi^u$ has a unique $\mathscr{h}$-conformal measure.
\end{corollaryA}

The previous two results are evidence that parabolic systems may not have many more than one conformal measure; we can ask

\begin{question}
If $\Phi$ is a minimal parabolic flow and $J$ is non-negative multiplicative cocycle, does there exist a unique $J$-conformal measure.
\end{question}

This question is particularly interesting for unipotent flows, in view of the very general rigidity results due to Ratner \cite{Ratner1991}, and others derived from her work.

\paragraph{Equilibrium States}

The proof of the above \hyperlink{theoremA}{Theorem A} relies on the study of equilibrium states for the time-one map of the corresponding hyperbolic flow. Let us recall some basic definitions and refer the reader to \cite{EqStatesCenter} and references therein for a more through discussion.

For a continuous endomorphism of a compact metric space $f:M\to M$ we denote by $\PTM{f}{M}$ the set of $f$-invariant probability measures, equipped with the $\omega^{\ast}$-topology. Given a continuous real valued map  $\varphi:M\to \Real$ (the potential), the topological pressure of the system is given as \cite{WaltersPres}
\begin{equation}\label{eq:variationalprinciple}
\Ptop(\varphi)=\sup_{\mathclap{\nu\in \PTM{f}{M}}}\{h_{\nu}(f)+\int \varphi d\nu\} 
\end{equation}
where $h_{\nu}(f)$ denotes the metric entropy of $f$ with respect to $\nu$. We denote
\begin{equation}
 	\Eq_f=\{\mu\in\PTM{f}{M}:\Ptop(\varphi)=h_{\mu}(f)+\int \varphi d\mu\};
 \end{equation} 
elements in $\Eq_f$ are called \emph{equilibrium states} (for the system $(f,\varphi)$). 

Similarly, if $\Phi=(\Phi_t)_t$ is a flow we denote $\PTM{\Phi}{M}=\cap_{t}\PTM{\Phi_t}{M}$ the set of its invariant measures; the topological pressure of the system $(\Phi,\varphi)$ is given as 
\begin{equation}
\Ptop(\Phi, \varphi)= \sup_{\mathclap{\nu\in \PTM{\Phi}{M}}}\{h_{\nu}(\Phi_1)+\int \varphi d\nu\}, 
\end{equation}
and we define $\Eq(\Phi,\varphi)$ analogously. One can show (cf. proposition $6.2$ in \cite{EqStatesCenter}) that for a (mixing) Anosov flow one has $\Eq(\Phi,\varphi)=\Eq(\Phi_1,\tilde{\varphi})$ where
\[
	\tilde{\varphi}(x)=\int_0^1 \varphi(\Phi_t x) dt.
\]

The next theorem is a central piece in smooth ergodic theory.

\begin{theorem*}[\cite{SinaiMarkov,SRBattractor,Bowen1974}]
    Let $\Phi$ be a transitive Anosov flow of class $\mathcal{C}^1$ and $\varphi:M\to\Real$ be a H\"older function, 
    Then $\Eq(\Phi,\varphi)$ consists of a unique element $\muf[(\Phi,\varphi)]$. 
\end{theorem*}

Here we are able to give the following strong characterization of $\muf[\Phi,\varphi]$ in terms of conditional measures along unstables. Let us recall that given a probability measure $\mm$ on $M$ a measurable partition (in the sense of Rokhlin \cite{Rokhlin}) $\xi$ is said to be 
\begin{itemize}
 	\item increasing if $\Phi_1\xi<\xi$ (i.e.\ for $\mm\aep(x)$ the atom $\Phi_1(\xi(x))$ consists of atoms of $\xi$);
 	\item subordinated to $\Fu$ is for $\mm\aep(x)$ it holds $\xi(x)\subset \Wu{x}$;
 	\item an SLY\footnote{By Sinai, Strelcyn, Ledrappier and Young} partition if it is increasing, subordinated to $\Fu$, and for $\mm\aep(x)$ the atom $\xi(x)$ contains a relative neighborhood of $x$ in $\Wu{x}$. See 
 \end{itemize} 
It is by now classical (and holds in much more generality than the case we are considering here) that for Anosov flows there exist SLY partitions or arbitrarily small mesh. See for example \cite{LedStrelcyn}.

We can prove the following.

\begin{theoremB}\hypertarget{theoremB}{}
Assume the same hypotheses of \hyperlink{theoremA}{Theorem A}. Then there exists a family of measures $\{\nu^u_x\}_{x\in M}$, where each $\nu^u_x$ is a Radon measure on $\Wu{x}$, such that for any probability $\mm$  on $M$ (not necessarily invariant under any $\Phi_t$) whose conditionals in some $\mm$ SLY partition $\xi$ are of the form
\[
	\mm^{\xi}_x=\frac{\nu^u_x}{\nu^u_x(\xi(x))},
\]
necessarily satisfies $\mm=\muf[\Phi,\varphi]$. 
\end{theoremB}

It is worth to emphasize that above we are not requiring the measure to be invariant (or quasi-invariant), and we do not assume any information in the transverse direction; \hyperlink{theoremB}{Theorem B} evidences a strong rigidity property of equilibrium states for hyperbolic systems.

\paragraph{Comparison with existing literature}

\hyperlink{theoremA}{Theorem A} was established originally by M. Babillot and F. Ledrappier \cite{geodesicbabillotledrappier} in the case of (Abelian covers of) the geodesic flow in a hyperbolic closed manifold, and later proven by a more geometrical argument by B. Schapira \cite{SCHAPIRA2004} in the same setting. Our method has some similarities with Schapira's, but is less dependent on non-trivial geometrical considerations for the geodesic flow, and is more direct. This allows us to consider general (codimension one) Anosov flows and diffeomorphisms. Of course, the tailored arguments for the geodesic flow give additional information in that case, and are generalizable for some non-compact hyperbolic manifolds; see \cite{schapirafrench}, \cite{Paulin2015}. Applicability of the ideas in this paper to the non-compact case remains to be investigated. 

The case of diffeomorphisms (Anosov) can be deduced from the work of Babillot and Ledrappier above cited, using the symbolic model of the map to translate the problem to subshifts of finite type, and then applying the powerful tools of spectral theory of transfer operators in this setting. For this, one needs to consider only the positive part of the shift (that is, the unstable sets) since the behavior of invariant measures on the other ``transverse'' direction (stable) is determined by the former. This is a particularity of the symbolic model, and has been exploited extensively. For example, in \cite{series1980} C. Series obtains the analogue of \hyperlink{theoremA}{Theorem A} for the lifted foliations to the symbolic model, and classifies the Borel equivalence relations induced by these. This was pushed further in \cite{symmetricgibbs}, still in the symbolic setting. To recover the same result for the diffeomorphism (uniqueness of the quasi-invariant transverse measure) one needs to control an additional direction; this is non-trivial but could be done, particularly in cases where the measure is known to be well behaved (as the entropy maximizing measure or the SRB). 

The case of flows if different, since typically (for example, for mixing Anosov flows) there is no possible symbolic model for the time-$t$ maps, and hence one cannot reduce their study to subshifts. What one has is a suspension of a subshift covering the flow, but this suspension is via a non-cohomologus to constant function, which depends on some arbitrary choices, and therefore cannot be used to give much information about the individual time-$t$ maps. See Bowen's article \cite{SymbHyp}.

We follow here a completely different approach, without reducing the problem to the symbolic flow, and this allows us to control the transverse direction as well. We remark also that in Babillot and Ledrappier's work (and also, Schapira's) the transverse direction is controlled using the rigid geometry of the map, and that's probably the main reason why their method cannot also be applied to non-symmetric Anosov flows. 

Finally, let us point out that \hyperlink{theoremB}{Theorem B} is new even for (linear!) Anosov diffeomorphisms, as we are not requiring invariance of the measure.

\paragraph*{Acknowledgments} The authors thank Barbara Schapira for bringing to our attention the references \cite{geodesicbabillotledrappier,SCHAPIRA2004,schapirafrench,Paulin2015}. We also thank the referee for her/his careful reading, and for pointing out to us some mistakes and omissions.


\section{Measures along strong unstable leaves and characterization of equilibrium states}\label{sec:Basic_theory}

In this part we give the preliminaries necessary for the proof \hyperlink{theoremA}{Theorem A} and \hyperlink{theoremB}{Theorem B}.

\smallskip 

From now on $\Phi=(\Phi_t:M\to M)_{t\in\Real}$ denotes a transitive hyperbolic flow of codimension-one that is not a suspension, with associated invariant decomposition $TM=E^{s}\oplus E^c\oplus E^{u}$, with $E^{u}$ tangent to the flow lines of the horocyclic flow $\Phi^u$. The bundles $E^{s}, E^{u}, E^{cs}=E^{c}\oplus E^{s}, E^{cs}=E^c\oplus E^{u}$ are all integrable to flow-invariant foliations $\F^{\ast}$, where the superscript $\ast$ coincides with the one of the bundle that it integrates. In particular, $\Fu$ is the foliation by orbits of $\Phi^u$. Each leaf of $\Fcs,\Fcu$ is saturated by leaves of $\Fs, \Fc$ and $\Fu, \Fc$, respectively.

\smallskip

\noindent\textbf{Notation:} If $A\subset M, \ep>0$ and $\F^{\ast}=\{W^{\ast}(x)\}_{x\in M}$ is one of the invariant foliations, let
\[
	W^{\ast}(A,\ep):=\{y\in M:\exists x\in A/ d_{\mathcal{\F^{\ast}}}(x,y)<\ep\},
\]
where $d_{\mathcal{\F^{\ast}}}$ is the intrinsic distance in the corresponding leaf. If $A=\{x\}$ we write $W^{\ast}(x,\ep)=W^{\ast}(\{x\},\ep)$.  

\smallskip 
 
Since the foliations $\Fu, \Fcs$ are transverse, there exists some $\mathrm{c}_{\scriptstyle geo}>0$ such that
\[
	x,y\in M, d(x,y)<\ep\leq\mathrm{c}_{\scriptstyle geo}\Rightarrow \# \Wu{x,2\ep}\cap \Wcs{y,2\ep}=1.
\]

For a submanifold $L\subset M$ denote $\RM[L]$ the set of Radon measures (for the induced topology) on $L$, and for $\F^{\ast}$ consider
\begin{align*}
&\mathrm{Rad}(\F^{\ast}):=\bigsqcup_{L\in\F^{\ast}} \RM[L]\\
&\mathrm{Meas}(\F^{\ast}):=\{\nu:M\rightarrow \mathrm{Rad}:\nu_x:=\nu(x)\in \RM[W^{\ast}(x)]\forall\ x\in M\}.
\end{align*}

We have the following.

\begin{proposition}[Proposition $6.1$ in \cite{EqStatesCenter}]\label{pro:medidasinestables}
There exists $\mu^u\in\mathrm{Meas}(\Fu)$ satisfying:
\begin{enumerate}
	\item[a)] $\mu^u_x$ is non-atomic and $\supp(\mu^u_x)=\Wu{x}$.
	\item[b)] $\forall t\in\Real, x\in M$, it holds 
	\[\Phi_{-t}\mu^u_{\Phi_t(x)}=e^{\Ptop(\varphi) t-\int_0^t \varphi(\Phi_s(\cdot))ds}\mu^u_x.
	\]
	\item[c)] The map $x\to \mu^u_x$ is weakly continuous, meaning: given $y\in \Wcs{x,\mathrm{c}_{\scriptstyle geo}}$ and $\hs_{x,y}:\Wu{y}\to\Wu{x}$ the locally defined Poincare' map determined by the center unstable foliation, it follows that for any  $A\subset \Wu{x}$ relatively open and pre-compact it holds
	\[
	\mu_{y}^u(\hs_{x,y}(A))\xrightarrow[y\to x]{}\mu^u_x(A).
	\]
	If moreover $A\subset \Wu{x, \mathrm{c}_{\scriptstyle geo}}$, then the convergence is uniform in $x$.
\end{enumerate}
\end{proposition} 

With this we were able to establish the next theorem.

\begin{theorem}[Theorem $6.3$ in \cite{EqStatesCenter}]\label{thm:Bforrankone}
If $\Phi$ is $\mathcal{C}^2$ Anosov flow with minimal unstable foliation, then the family $\mu^u$ of the previous theorem satisfies: if $\mm\in\PM[M]$ satisfies
\begin{enumerate}
 	\item $\mm$ is $\Phi_t$-invariant for some $t\neq 0$, and
 	\item $\mm$ has conditionals absolutely continuous with respect to $\mu^u$,
 \end{enumerate}
 then $\mm=\muf[(\Phi,\varphi)]$. 
\end{theorem}

From now on we denote $f=\Phi_1$ and fix $\varphi:M\to\Real$ H\"older potential together with its corresponding measures $\mu^u$ given in \cref{pro:medidasinestables}. We denote
\[
	h_{\varphi}(x)=e^{\tilde{\varphi}(x)},
\]
and given $x\in M$ define $\Delta_x^u:\Wu{x}\to\Real_{>0}$,
\begin{equation}\label{eq:JacobianoDelta}
y\in\Wu{x}\Rightarrow \Delta_x^u(y):=\prod_{k=1}^{\oo}\frac{h_{\varphi}(f^{-k}y)}{h_{\varphi}(f^{-k}x)}.
\end{equation}
Since $\tilde{\varphi}$ is H\"older it follows that $\Delta_x^u$ is continuous, and converges to one as $y\mapsto x$.

\begin{definition}\label{def:medidasproyectivas}
For $x\in M$ let $\nu^u_x=\Delta_x^u\ \mux$.
\end{definition}

Clearly $\nu^u_x\in\RM[\Wu{x}]$ and satisfies the following properties.
\begin{enumerate}
	\item[a)] If $y\in \Wu{x}$ we have $\nu^u_y=c(y,x)\cdot \nux$ for some constant $c(y,x)=\Delta_y^u(x)>0$, hence $\{\nu^u_y\}_{y\in\Wu{x}}$ defines a projective class of measures $[\nu^u_x]$ in $\Wu{x}$.
	\item[b)] $f^{-1}\nux[fx]=e^{\Ptop(\varphi)-\tilde{\varphi}(x)}\nux$.
\end{enumerate}

\smallskip

\begin{definition}\label{def:compatiblemeasures}
A family of measures $\zeta^{cs}\in\mathrm{Meas}(\Fcs)$ is said to be compatible with $(\varphi,\mu^u)$ if its quasi-invariant for $\Phi^u$ with Jacobian given by the family $\mathscr{h_{\varphi}}$ (cf.\@ \cref{eq:hjacobiano}).  
\end{definition}

We proved before (Section $3.2$ of \cite{EqStatesCenter}) that given $\zeta^{cs}$ compatible with $(\varphi,\mu^u)$   
one can construct a probability measure $\mm$ on $M$ that is given locally as follows. For $W=\Wcs{x_0,\ep}, 0<\ep\leq\mathrm{c}_{\scriptstyle geo}$,
\[
	\mm\approx \int_W \nux[x]\ d\zeta^{cs}_{x_0}(x). 
\]
Compatibility is used to show that the definition does not depend on the chosen center stable disc, and therefore one can glue these local measures into $\mm$.

\begin{remark}
Moreover, if $\zeta^{cs}$ is a quasi-invariant measure for $f$ as in \cref{pro:quasinvariancecocycle}, then the resulting $\mm$ is an equilibrium state for $(f, \tilde{\varphi})$, and therefore the unique equilibrium state for the system $(\Phi,\varphi)$. This does not require codimension-one and was noted in our previous article.
\end{remark}

We point out that there exists a family $\mu^{cs}\in\mathrm{Meas}(\Fcs)$ satisfying the quasi-invariance condition refereed in the previous remark. The equilibrium state $\muf[\Phi,\varphi]$ can be constructed using this family.

\section{Unique ergodicity for the conditional measures} 
\label{sec:unique_ergodicity}

Let $\mathrm{c}_{\scriptstyle int}>0$ be much smaller than $\mathrm{c}_{\scriptstyle geo}$, and for $x\in M$ define
\begin{equation}
\label{eq:Ix} I_x:=\Phi^u_{[-\mathrm{c}_{\scriptstyle int},\mathrm{c}_{\scriptstyle int}]}(x).
\end{equation}
The set $I_x$ will be called a $u$-interval. Since the conditionals $\nux$ are non-atomic (cf.\@ \cref{pro:medidasinestables}), we deduce:

\begin{lemma}\label{lem:fronteraIxesnula}
It holds $\forall x\in M, \nu^u_x(\partial I_x)=0$.
\end{lemma}

Let $\mathcal{I}=\{I_x\}_{x\in M}$ and observe that the family $\mathcal{I}$ has the following continuity property:
\[
 I_y\xrightarrow[y\to x]{\mathrm{Haus}} I_x
 \]  
where the above convergence is in the Hausdorff topology. We also have the following.

\begin{lemma}\label{lem:continuidadnuxIx}
The functions $x\to \mux(I_x)$ and $x\to\nux(I_x)$ are continuous.
\end{lemma}

\begin{proof}
By \cref{eq:JacobianoDelta} it suffices to prove the first part, which follows by $c)$ of \cref{pro:medidasinestables}, the lemma above, and Alexandrov's theorem.
\end{proof}

Recall that $f=\Phi_1$ denotes the time-one map of a (topologically mixing) codimension-one Anosov flow. For $\h\in \CM[M]$ we define the operator 
\begin{equation}\label{eq:Rnh}
R_n\mathfrak{h}(x)=\frac{1}{\nux[f^n(x)](f^n(I_x))}\int_{f^n(I_x)}\h(t)d\nux[f^n(x)](t)=\frac{1}{\nux(I_x)}\int_{I_x}\h\circ f^n(t)d\nux(t).	
\end{equation}

From now on we fix $\muf$ an equilibrium state for $(f,\tilde{\varphi})$ constructed as explained in the previous section. 

\medskip

\noindent\textbf{Notation.} We write $W^{\ast}_{\scriptstyle \mathrm{loc}}(x)=W^{\sigma}(x,\mathrm{c}_{\scriptstyle geo})$. If $B$ is foliation box of diameter less than $\mathrm{c}_{\scriptstyle geo}$, the connected component of $W^{\ast}_{\scriptstyle \mathrm{loc}}(x)\cap B$ that contains $x$ is denoted by $W^{\ast}(x, B)$.

\smallskip 
 
We start by establishing the following.

\begin{proposition}\label{pro:equicontinuidad}
For every $\h\in \CM$ the family $\{R_n\h\}_n$ is equicontinuous and uniformly bounded.
\end{proposition}
\begin{proof}
As the family $\{R_n\h\}_n$ is clearly bounded by $\normC{\h}{0}$ we only need to show that it is equicontinuous. To this end, given two points $x,y$ with $d(x,y)<\mathrm{c}_{\scriptstyle geo}$ we can choose $z$ and $w$ so that $z\in \Wuloc{x}\cap \Wcsloc{y}$ and $w\in \Wcloc{z}\cap \Wsloc{y}$. See \cref{fig:comparaumeasure} below. To show equicontinuity we will make comparisons between $x$ and $z$, $z$ and $w$ and $w$ and $y$.

\begin{figure}[H]
	\centering
		\includegraphics[scale=0.9]{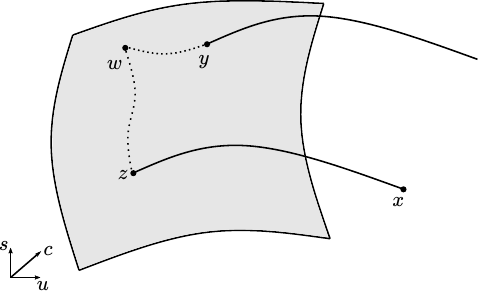}
	\caption{Diagram for the proof.}
	\label{fig:comparaumeasure}
\end{figure}

 Since $z\in\Wu{x}$ we have $\nux=c(x,z)\nux[z]$, where $c(x,z)$ converges uniformly to one as $z$ approaches $x$ (cf.\@ the remark after the definition of the measures $\nux$). Therefore
\begin{align*}
|R_n\h(x)-R_n\h(z)|&=\left|\frac{1}{\nux(I_x)}\int_{I_x} \h\circ f^n(t)d\nux(t)-\frac{1}{\nux(I_z)}\int_{I_z} \h\circ f^n(t)d\nux(t)\right|\\
&\leq\norm{\h}_{C^0}\int\left|\frac{\one_{I_x}(t)}{\nux(I_x)}-\frac{\one_{I_z}(t)}{\nux(I_z)}\right|d\nux(t)\leq \frac{2\norm{\h}_{C^0}}{\nux(I_x)}\cdot\nux(I_x\triangle I_z)
\end{align*} 
and the later term can be taken arbitrarily small if $z$ is close enough to $x$, by \cref{lem:continuidadnuxIx,lem:fronteraIxesnula}.

The comparison between $w$ and $y$ follows similar lines, using that $\Jacu[w,y]$ converges uniformly to $1$ as $w$ approaches $y$ (cf.\@ \cref{pro:quasinvariancecocycle}). Finally, for the comparison between $z$ and $w$ we observe that the center holonomy (i.e.\@ the flow) sends $\nu_z^u$ to $g_{z,w}\nu^u_w$, where $g_{z,w}$ is a continuous positive function that converges uniformly to one as $z\mapsto w$. Therefore, we get 
\begin{align*}
&\left|\frac{1}{\nux(I_z)}\int_{I_z}\h\circ f^n(t)d\nux[z](t)-\frac{1}{\nux[y](I_w)}\int_{I_w}\h\circ f^n(t)d\nux[w](t)\right|\\
&=\left|\frac{1}{\nu^u_z(I_z)}\int\Big(\one_{I_z}(t)\h\circ f^n(t)-\frac{\nux[z](I_z)}{\nux[w](I_w)}\one_{\hc_{z,w}(\tilde{I}_z)}(t)\h\circ f^n(\hc(t))\cdot g_{z,w}(t)\Big)d\nux[w](t)\right|	
\end{align*}

As $z\mapsto w$ we have $\frac{\nux[z](I_z)}{\nux[w](I_z)}\rightrightarrows 1$ as well. Since $f$ is part of the flow,  we conclude that $|\h\circ f^n(\hc(t))-\h\circ f^n(t)|$ can be taken uniformly small for $z$ close to $w$, independently of $n$. Altogether we have shown that  for every $n$, $|R_n\h(z)-R_n\h(w)|$ is small whenever $w\in\Wc{z}$ is close to $z$, and this concludes the proof of equicontinuity.
\end{proof}

\subsection{SOT convergence of \texorpdfstring{$(R_n)_n$}{Rn}.} Now we will show the central part of our argument: the family of positive operators $(R_n:\CM\to\CM)_n$ converges in the strong operator topology to some multiple of the identity. More precisely, there exists some probability measure $\mathrm{p}^u$ such that 
\[
	\h\in\CM\Rightarrow \normC{R_n\h -\int \h d\mathrm{p}^u}{0}\xrightarrow[n\to\oo]{}0. 
\]
Due to Riesz-Markov representation theorem, this amounts to establishing the following.

\begin{theorem}\label{thm:uniformconvergenceofRnh}
For every $\h\in \CM$ the sequence $(R_n\h)_n$ converges uniformly to a constant $a$.
\end{theorem}

We first determine the value $a$. Define $\Omega$ as the set of all sequences
$\{\Theta^m_n\}_{m,n\geq 0}$ of functions $\Theta^m_n:M\rightarrow \mathrm{Rad}(\Fu)$ satisfying
\begin{enumerate}
\item $\Theta^m_{n,x}:=\Theta^{m,n}(x)$ is a probability measure on $f^m(I_x)$, and
\item given $\epsilon>0$ there exists $N_{\epsilon}>0$ such that for $n,m\geq N_{\epsilon}$ it holds for every $x\in M$
\begin{equation}\label{eq:Thetameasures}
R_{n+m}\h(x)=\int R_{n}\h(y)d\Theta^m_{n,x}(y)+E_{n,m}(x)
\end{equation}
where $|E_{n,m}(x)|<\epsilon$.
\end{enumerate}
Assume for the moment that $\Omega$ is non-empty and take $\{\Theta^m_n\}_{n,m\geq 0}\in \Omega$. Define 
\[
a_{n}:=\inf_{x}R_n\h(x)
\]
and observe that by \eqref{eq:Thetameasures}
\begin{align*}
a_{n+m}\geq a_n-\sup_{x}|E_{n,m}(x)|.
\end{align*}

Let $a:=\limsup_n a_n,b=\liminf_n a_n$; these are finite numbers in virtue of \cref{pro:equicontinuidad}. Given $\epsilon>0$ take $N_{\epsilon}$ such that $n,m\geq N_{\epsilon}$ implies $\sup_{x}|E_{n,m}(x)|<\epsilon$: since any neighborhood of $a$ contains points of $(a_n)$, by the previous formula we deduce that $|b-a|<2\epsilon$, and as $\epsilon$ is arbitrary we have that $a=b$, therefore
\[
a=\lim_n a_n.
\]

Now that we have our candidate for constant function, we will proceed as follows: consider $\g\in \CM$ an accumulation point of some sub-sequence $(R_{n_k}\h)_{k}$ (which exists by equicontinuity of the family). Given an open set $U\subset M$, we will find $\Theta\in\PM$ with $U\subset \supp(\Theta)$ and $\int_U (\g-a) d\Theta=0$. Since $\g(x)\geq a$ for every $x$, this will imply the existence of some $x_U\in U$ such that  $g(x_U)=a$ on $U$. As $U$ can be chosen in a basis of the topology on $M$, we will deduce that $g=a$, thus showing $\lim_n  \normC{R_n\h-a}{0}=0$.

The remainder of this part is devoted to establishing this result.

\smallskip

Take $U\subset M$ open box of the form $U\approx \Wc{\Ws{x_0,\delta},\delta}\times \Wu{x_0,\delta}$ with $\delta$ much smaller than $\mathrm{c}_{\scriptstyle geo}$.

\medskip 

\noindent\textbf{Notation.} For $0<\rho<< \mathrm{c}_{\scriptstyle geo}$ and $S\subset M$ we denote by $\partial_{\rho} S$ the $\rho$-neighborhood of the boundary of $S$. If $U\approx \Wc{\Wu{x_0,\delta},\delta}\times \Ws{x_0,\delta}$ is a box then its $cu$-boundary of size $\rho$ is the set $\partial^{cu}_\rho U$ corresponding to $\Wc{\Wu{x_0,\delta},\delta}\times \partial_{\rho} \Ws{x_0,\delta}$.  

\smallskip

\begin{lemma}\label{lem:boundU}
If $\muf(\partial U)=0$ there exists $\mathrm{C}_U>0$ such that for every $x\in M$
\[
\liminf_{m\mapsto\oo}\frac{\nu_{f^mx}(U\cap f^m(I_x))}{\nu_{f^mx}(f^m(I_x))}>\mathrm{C}_U>0.
\]
\end{lemma}

\begin{remark}
Since $(f,\muf)$ is a Kolmogorov system, the above inequality holds for $\muf\aep(x)$, however for our purposes it is important to establish the validity for all $x\in M$.
\end{remark}

\begin{proof}
Start observing that $\frac{\nu_{f^mx}(U\cap f^m(I_x))}{\nu_{f^mx}(f^m(I_x))}=\frac{\nu_{x}(f^{-m}U\cap (I_x))}{\nu_{x}((I_x))}$.

Assume by means of contradiction that the conclusion is not true. Since $\inf_x \nu^u_x(I_x)>0$ we have that for every $\tau>0$ there exists $x_{\tau}$ and  $n_{\tau}$ such that for every $n\geq n_{\tau}$ it holds
\[
\nu_{x_{\tau}}(I_{x_{\tau}}\cap f^{-n}U)\leq \tau. 
\]
Choose $\rho>0$ so that $\muf(\partial^{cu}_\rho U)$ is much smaller than $\tau$, set $d_U=\diam U$ and define 
\[
	B_x=\bigcup_{\mathclap{y\in \Wc{I_x, 5d_U}}}\ I_y
\]
We want to compare the measures $\muf(B_x\cap f^{-n}U)$ and $\nu^u_x(I_x\cap f^{-n}U)$. Divide the connected components of $B_x\cap f^{-n}U$ into two types, $\Good_x^n$ and $\Bad_x^n$ where $Z\in \Good_x^n\Leftrightarrow$ for every $z\in Z$ it holds  
\begin{enumerate}
\item $f^n(\Ws{z, Z})\subset \Ws{f^nz, U}$;
\item $f^n(\Wu{z, Z}) \supset \Wu{f^nz, U}$.
\end{enumerate}
Observe that as $f$ preserves center lengths the second condition above implies  
\begin{enumerate}
	 \item[3.] $f^n(\Wcu{z, Z}) \supset \Wcu{f^nz, U}$
\end{enumerate}
It follows that if $Z\in \Bad^n_x$ then either
\begin{enumerate}
\item $Z\subset \partial_{\rho}^{cu} U$, or
\item $f^n(I_x\cap f^{-n}U\cap Z)\not \supset \Wu{f^nz, U}$. 
\end{enumerate}
Observe that we can write, 
\begin{align*}
\MoveEqLeft\frac{\muf(f^{-n}U\cap B_{x})}{\nux(f^{-n}\cap I_x)}=\frac{\muf(\bigcup_{Z\in \Good_x^n} Z)}{\nux(f^{-n}\cap I_x)}\left(1+\frac{\muf(\bigcup_{Z\in \Bad_x^n}Z)}{\muf(\bigcup_{Z\in \Good_x^n}) Z}\right)
\shortintertext{which by the local product structure of $\muf$ implies, for some uniform $\mathrm{c}_{\scriptstyle prod}$,} 
&\frac{\muf(\bigcup_{Z\in \Good_x^n} Z)}{\nux(f^{-n}\cap I_x)}\left(1+\frac{\muf(\partial^{cu}_{\rho} U)+\mathrm{c}_{\scriptstyle prod}\cdot\nux(\partial_{\lambda^{-n}}2\mathrm{c}_{\scriptstyle int})}{\muf(\bigcup_{Z\in \Good_x^n}) Z}\right).
\end{align*}
Using again the local product structure we deduce the existence of $\tilde{\mathrm{d}}>0$ satisfying: for every $x$ there exists $m_x$ so that for $n\geq m_x$,
\[
	\muf(f^{-n}U\cap B_{x})\leq \tilde{\mathrm{d}}\cdot \nux(I_x\cap f^{-n}U). 
\]
Putting all the pieces together we finally conclude that for every $\epsilon>0$ there exist $x_{\ep}\in M$ and $n_{\ep}'\in\Nat$ such that if $n\geq n_{\ep}'$ then 
\begin{equation*}
\muf(B_{x_{\tau}}\cap f^{-n}U)<\epsilon.
\end{equation*}
Since the system is mixing, we then deduce that for $n\geq n_{\ep}''$,
\[
 \muf(B_{x_{\tau}})\cdot \muf(U)< 2\ep. 	
\]
We will show below that $\inf_{x} \muf(B_{x})>0$: but then we get that $\muf(U)=0$, contradicting the fact that $\muf$ has full support. To finish the proof we establish the following.

\smallskip 
 
\noindent\textbf{Claim:} If $\lambda$ is a probability on $M$ with full support, then $\inf_{x}\lambda(B_x)>0$.

\smallskip

Suppose not: then we can find a converging sequence $x_n\xrightarrow[n\mapsto\oo]{}x$ such that the corresponding sets $B_{x_n}$ converge to $B_x$ in the Hausdorff topology, and furthermore
\[
\lambda(B_{x_n})\xrightarrow[n\mapsto\oo]{} 0. 
\]
Take $D(x,r)\subset B_x$ open disc of radius $r$, centered at $x$, and such that $\lambda(\partial D(x,r))=0$: as 
\begin{align*}
D(x_n,r)\xrightarrow[n\mapsto\oo]{\mathrm{Haus}}D(x,r)
\end{align*}
it follows that
\[
\lambda(D(x_n,r))\xrightarrow[n\mapsto\oo]{}\lambda(D(x,r)),
\]
and hence $\lambda(D(x,r))=0$, contrary to our assumption that $\lambda$ has full support.
\end{proof}

We want to sub-divide $f^m(I_x)$ in intervals that are well positioned with respect to the connected components of $f^m(I_x)\cap U$; for this we will use the following lemma (cf.\@ lemma $3.2$ in \cite{Marcus1975}).

\begin{lemma}\label{lem:Marcus}
\label{lem:marcus} Given $\ep>0$ there exists $n_{\ep}\in \Nat$ such that for every $x\in M$, the set
\[
	\{\Phi^u_{t 2\mathrm{c}_{\scriptstyle int}}(x):t=0,\cdots, n_{\ep}\}
\]
is $\ep$-dense.
\end{lemma}

To cover $f^m(I_x)$ we use the following algorithm. Denote $\mathrm{d}_U=\diam U$ and choose $0<\ep< \frac{\mathrm{d}_U}{5}$. For $x\in M$ consider $\{C_{x}^{(m)}(l)\}_{l=0}^{o(x,m)}$ the set of connected components of $f^m(I_x)\cap U$, ordered according to the orientation of $f^m(I_x)$; note that since $\mathrm{d}_U$ is small with respect to $\mathrm{c}_{\scriptstyle int}$, if $y\in C_{x}^{(m)}(l)$ then $I_y\cap U=C_{x}^{(m)}(l)$. Choose $y_{x,0}^{(m)}\in C_{x}^{(m)}(0)$ and for $n\in \Z$ let $y_{x,n}^{(m)}=\Phi^u_{n2\mathrm{c}_{\scriptstyle int}}(y_{x,0}^{(m)})$. Observe that by applying the above lemma to the family $\{y_{x,n}^{(m)}\}_{n\geq 0}$ we get 
\begin{equation}\label{eq:proporcionplacafmU}
\forall l, \#\{l\leq n\leq l+2n_{\ep}: y_{x,n}^{(m)}\in U\}\geq n_{\ep}.
\end{equation}
After choosing $y_{x,0}^{(m)}$ consider $\overline{k}_x^{m}=\overline{k}_x^m(y_{x,1}^m), \underline{k}_x^m=\underline{k}_x^m(y_{x,1}^m)$ the largest natural numbers such that $I_{y_{x,\overline{k}_x^m}^{(m)}}\subset f^m(I_x), I_{y^{(m)}_{x,-\underline{k}_x^m}}\subset f^m(I_x)$ and let
\begin{align}
&\Good^m_x:=\bigcup_{l=-\underline{k}_x^m}^{\overline{k}_x^m} I_{y_{x,l}^{(m)}}\\
&\Bad^m_x:=f^m(I_x)\setminus\Good^m_x 
\end{align}

Note that 

\begin{equation}\label{eq:medidamalos}
\frac{\nu^u_{f^mx}(\Bad^m_x)}{\nu^u_{f^m(x)}(f^m(I_x))}\leq \frac{\nu^u_{f^m(x)}(\partial_{4\mathrm{c}_{int}} f^m(I_x))}{\nu^u_{f^m(x)}(f^m(I_x))}=\frac{\nu^u_x(f^{-m}(\partial_{4\mathrm{c}_{int}} f^m(I_x))}{\nu^u_{x}(I_x)}.
\end{equation}

\begin{lemma}\label{lem:medidamalos}
$\frac{\nu^u_x(f^{-m}(\partial_{4\mathrm{c}_{int}} f^m(I_x))}{\nu^u_{x}(I_x)}\xrightarrow[m\to\oo]{}0$ uniformly in $x$.	
\end{lemma}	

\begin{proof}
Assume that this were not the case: then we could find $\zeta>0$, a converging sequence $x_k\xrightarrow[k\mapsto\oo]{} x$ and corresponding $m_k\xrightarrow[k\mapsto\oo]{}\oo$ such that 
\begin{itemize}[leftmargin=*]
\item $\frac{\nu^u_{x_k}(f^{-m_k}(\partial_{4\mathrm{c}_{int}} I_{f^{m_k}(x_k)}))}{\nu^u_{x_k}(I_{x_k})}\geq \zeta$.
\item $\clo{I_{x_k}}\xrightarrow[k\mapsto\oo]{\mathrm{Haus}}\clo{I_x}$.
\end{itemize}

For $k$ large the set $\mathrm{hol}^{cs}_{x, x_k}\left(f^{-m_k}(\partial_{4\mathrm{c}_{int}} I_{f^{m_k}(x_k)})\right)$ 
is contained in a $\delta_k$-neighborhood of $\partial I_x$ inside $\Wu{x}$. Since $f^{-1}$ is uniformly contracting on unstable leaves, it follows that one can take $\delta_k\mapsto 0$, and using that the measures $\nu^u_{x_k}$ and $\nu^u_{x}$ are comparable we finally deduce the existence of $\zeta'>0$ such that 
\[
\nu^u_{x}(\partial_{\delta_k} I_{x})\geq \zeta'.
\]

This contradicts the fact that $\nu^u_x(\partial I_x)=0$.
\end{proof}

On the other hand, by \cref{lem:boundU,eq:proporcionplacafmU} and since the sizes of $\{\nux(I_x)\}_{x\in M}$ are uniformly comparable, we deduce the following.

\begin{lemma}
\label{lem:medidaplacasUIx} There exists $\mathrm{c}_{U}'>0$ such that 
\[
	\forall x\in M,\ \liminf_{m\to\oo}  \frac{\nux[f^mx](\{Z\in \Good^m_x:Z\cap U\neq \emptyset\})}{\nux[f^mx](\{Z\in \Good^m_x\})}\geq \mathrm{c}_U'>0.
\]
\end{lemma}

Now we prove that $\Omega\neq \emptyset$.

\begin{proposition}\label{pro:existenciathetameasurecasoequicontinuo}
There exist $\{\Theta^m_n\}_{m,n\geq 0}\in \Omega$ and $N(U)\in\Nat, \mathrm{C}_U>0$ such that for $n,m\geq N(U),x\in M$ it holds
\[
\Theta^m_n(U)\geq \mathrm{C}_U.
\]
\end{proposition}

\begin{proof}
Let us start by computing 
\begin{align*}
&R_{n+m}\h(x)=\frac{1}{\nux(I_x)}\int_{I_x}\h\circ f^{n+m}(t)d\nux(t)=\frac{1}{\nux[f^mx](f^m(I_x))}\int_{f^m(I_x)}\h\circ f^{n}(t)d\nu^u_{f^m(x)}(t)\\
&=\frac{1}{\nux[f^m(x)](f^m(I_x))}\sum_{l=-\underline{k}_x^m}^{\overline{k}_x^m}\int_{I_{y_{x,l}^{(m)}}}\h\circ f^nd\nu^u_{f^m(x)}
+\frac{1}{\nux[f^m(x)](f^m(I_x))}\int_{\Bad^m_x}\h\circ f^nd\nux[f^m(x)]\\
&=\frac{1}{\nux[f^m(x)](f^m(I_x))}\sum_{l=-\underline{k}_x^m}^{\overline{k}_x^m}  \frac{\nux[f^m(x)](I_{y_{x,l}^{(m)}})}{\nux[y_{x,l}^{(m)}](I_{y_{x,l}^{(m)}})} \int_{I_{y_{x,l}^{(m)}}}\h\circ f^nd\nux[y_{x,l}^{(m)}]\\
&+\frac{1}{\nux[f^m(x)](f^m(I_x))}\int_{\Bad_x^m}\h\circ f^nd\nux[f^m(x)].
\end{align*}

Denote $c^{(m)}_l(x)=\frac{\nux[f^m(x)](I_{y_{x,l}^{(m)}})}{\nux[f^m(x)](f^m(I_x))}$ and define 
\[
\Theta^m_{n,x}=\sum_{l=-\underline{k}_x^m}^{\overline{k}_x^m} c^{(m)}_l(x)\delta_{y_{x,l}^{(m)}}+\frac{\nux[f^m(x)](\ \cdot\ \cap \Bad_x^m)}{\nu^u_{f^m(x)}(f^m(I_x))};	
\]
this is a probability measure supported on $f^m(I_x)$, and by the computation above
\[
R_{n+m}\h(x)=\int R_n\h(t)d\Theta^m_{n,x}(t)+H^{(n)}_m(x)
\]
with $H^{(n)}_m(x)=\frac{1}{\nux[f^m(x)](f^m(I_x))}\int_{\Bad_x^m}\left(\h\circ f^n-R_n\h\right)d\nux[f^m(x)]$, which by \eqref{eq:medidamalos} and \cref{pro:equicontinuidad} satisfies
\begin{align*}
|H^{(m)}_{n}(x)|\leq 2\normC{\h}{0}\frac{\nu^u_{f^mx}(\Bad^m_x)}{\nu^u_{f^m(x)}(f^m(I_x))}\xrightarrow[m\to\oo]{}0
\end{align*}
uniformly in $x,n$. It remains to show that $\Theta^m_{n,x}(U)$ is uniformly bounded from below: we compute
\begin{align*}
\Theta^m_{n,x}(U)&\geq \sum_{l=-\underline{k}_x^m}^{\overline{k}_x^m} c^{(m)}_l(x)\shortintertext{which by \cref{lem:boundU,lem:medidaplacasUIx} satisfies for $m$ large}
&\geq \mathrm{C}_U:=C_Uc_U'. 
\end{align*}
\end{proof}

We are ready to establish the SOT convergence of $\{R_n\}_n$.

\begin{proof}[Proof of \cref{thm:uniformconvergenceofRnh}]
We will use the notation of the previous proposition. Letting $a_n=\inf R_n\h(x)$ it follows, since $\Omega\neq \emptyset$ that $\exists \lim_n a_n=\sup a_n= a$. Proposition \ref{pro:equicontinuidad} and Arzel\'{a}-Ascoli's theorem imply that $\{R_nh\}_n$ is pre-compact; we will show that it only accumulates on the constant function $a$, thus establishing the theorem. To this end, let $(n_k)_{k}$ be any sub-sequence such that $\{R_{n_k}h\}_k$ is convergent to some function $\g\in \CMc[M]$: necessarily $\min \g=a$.
	
Observe that for every $m$ it holds
\[
R_{n_k+m}\h(x)-a=
\sum_{l=-\underline{k}_x^m}^{\overline{k}_x^m} c^{(m)}_l(x) \left(R_{n_k}\h(y_{x,l}^{(m)})-a\right)+H^{(m)}_{n_k}(x)-a\left(1-\sum_{l=-\underline{k}_x^m}^{\overline{k}_x^m} c^{(m)}_l(x) \right).
\]

Take $x_{n_k+m}$ such that $R_{n_k+m}\h(x_{n_k+m})=a$: then  
\begin{equation}\label{eq:medidastheta1}
0=\sum_{l} c^{(m)}_l(x_{n_k+m})\left(R_{n_k}\h(z_{k,l}^{(m)})-a\right)+ H^{(m)}_{n_k}(x_{n_k+m})-a\left(1-\sum_l c^{(m)}_l(x_{n_k+m}) \right).
\end{equation}
Here the $z_{k,l}^{(m)}$ are associated to the points $x_{n_k+m}$. Now consider the sequence of  measures 
\[
\Theta^m_{(n_k)}=\Theta^m_{n_k,x_{n_k+m}},
\]
  
Again, using that $\frac{\nu^u_{f^mx}(\Bad^m_x)}{\nu^u_{f^m(x)}(f^m(I_x))}\xrightarrow[m\to\oo]{}0$ uniformly in $x$, we get that $\sum_l c^{(m)}_l(x)\xrightarrow[m\to\oo]{}1$ uniformly in $x$, and therefore
\[
H^{(m)}_{n_k}(x_{n_k+m})-a\left(1-\sum_l c^{(m)}_l(x_{n_k+m}) \right)\to 0
\]
as $m\to\infty$, uniformly for $n_k\to\infty$. By using the diagonal argument  we can find sub-sequences $\{n_{k'}\}_{k'}$ of $\{n_k\}_k$,  $\{m_j\}\subset \mathbb{N}$ and a measure $\Theta$ such that
\[
\lim_{k'\mapsto\oo}\lim_{j\mapsto\oo}\Theta^{(m_j)}_{n_{k'}}=\Theta.
\]
Note that by the previous proposition, $\Theta(U)\geq \mathrm{c}_U'>0$. Take limit as $m_j\to\infty$ and $n_k\to\infty$ simultaneously in equation \eqref{eq:medidastheta1} to get, on the one hand
\[
\int  \left(R_{n_k}\h-a\right) d\Theta^{(m_j)}_{n_k}\to 0
\]
and on the other,
\[\int  \left(R_{n_k}\h-a\right) d\Theta^{(n_k)}_{m_j}\to \int  \left(\g-a\right) d\Theta,
\]
so that 
\[\int  \left(\g-a\right) d\Theta=0.
\]
Finally, since $\min \g=a$, $\g-a\geq 0$ and thus
\[
0=\int  \left(\g-a\right) d\Theta\geq \int_U \left(\g-a\right) d\Theta\geq \inf_U (\g-a)\widetilde{C}(U) 
\]
which in turn implies that $\g-a$ has a zero in $U$. Observe that $U$ can be taken arbitrarily small (as long as $\muf(U)=0$), therefore we deduce that $\g\equiv a$ on $M$.

We have shown that the only accumulation point of $(R_n\h)_n$ is the constant function $a$, hence $\lim_n \normC{R_n\h-a}{0}=0$ as we wanted to show.
\end{proof}

\section{Proof of uniqueness: theorems A and B} 
\label{sec:proof_of_uniqueness_theorems_a_and_b}

Now we fix $\xi$ a SLY partition for the flow, and let 
\begin{equation}
	\mathcal{N}:=\bigcup_{n\geq 0}f^{-n}\bigcup_{x\in M} \partial\xi(x)
\end{equation}
By \cref{thm:Bforrankone} it holds that $\mm(\mathcal{N})=0$ for any equilibrium state $\mm$ of the system, in particular for $\mm=\muf$. For $n\in \Nat$ we consider also $\xi^n=f^n\xi$ (i.e.\@ $x\in\mathcal{N}, \xi^n(x)=f^n\xi(f^{-n}x)$). It is also no loss of generality to assume that for every $x\not\in N$ the atom $\xi(x)$ contains a neighborhood of $x$ inside $\Wu{x}$.

Given $\h\in \CM[M], x\in M\setminus \mathcal{N}$ define 
\begin{equation}\label{eq:Enh}
E_n\h(x)=\frac{1}{\nu^u_x(\xi^n(x))}\int_{\xi^n(x)} \h(t) d\nu^u_x(t)=\frac{1}{\nu^u_{f^{-n}x}(\xi(f^{-n}x))}\int_{\xi(f^{-n}x)} \h\circ f^n(t) d\nu^u_{f^{-n}x}(t).
\end{equation}

It follows that  
\begin{enumerate}
	\item $\sup_x |E_n\h(x)|\leq \normC{\h}{0}$;
	\item for every $\mm\in \PM[M]$ having conditionals given by $\{\nux\}_{x\in M}$ (in particular, for every equilibrium state by \cref{thm:Bforrankone}), $E_n\h$ is a version of the conditional expectation $\ie{\mm}{\h|\xi^n}$.
\end{enumerate}

We now argue similarly as we did in the previous part: given $x\in M\setminus \mathcal{N}$ and $m$ large, cover $f^m(\xi(x))$ as follows: consider $I_{z^{m}_{0,x}},\cdots, I_{z^{m}_{o(x,m),x}}\subset f^m(\xi(x))$ a maximal disjoint family, and call $\Good^m_x=\bigcup_{l=0}^{o(x,m)}I_{z^{m}_{0,x}}$, $\Bad^m_x=f^m(I_x)\setminus \Good^m$.

Observe that $\Bad^m_x\subset \partial_{2\mathrm{c}_{\scriptstyle int}}f^m\xi(x)$, and since $\nux(\partial \xi(x))=0$ we can argue as in \cref{lem:medidamalos} and deduce
\begin{equation}\label{eq:medidamalospuntual}
	\lim_{m\to \oo}\frac{\nux[f^mx](\cup_{Z\in \Bad^m_x}Z)}{\nux[f^mx](\cup_{Z\in \Good^m_x}Z)}=0.
\end{equation}

Now compute, as in the proof \cref{pro:existenciathetameasurecasoequicontinuo},
\begin{align*}
E_{n}\h(x)=\frac{1}{\nux(\xi^n(x))}\int_{\xi^n(x)} \h(t) d\nu^u_x(t)=\sum_{l=0}^{o(x,n)} \frac{\nux(I_{z^{n}_{l,x}})}{\nux(\xi^n(x))}R_n\h(z^{n}_{l,x})+\frac{1}{\nux(\xi^n(x))}\int_{\Bad^n_x}h(t) d\nux(t)\\
=\int R_n\h(t)d\Theta_{n,x}(t)+T_{n,x} 
\end{align*}
where $\Theta_{n,x}$ is a probability supported on $\xi^{n}(x)$ and $T_{n,x}$ converges to zero as $n\mapsto \oo$, due to \cref{eq:medidamalospuntual}. Taking an accumulation point of $\{\Theta_{n,x}\}$ and using the uniform convergence of $\{R_n\h\}_{n\geq 0}$ we deduce
\begin{equation}\label{eq:converEE}
\lim_n E_n\h(x)=a\quad \forall x\not\in\mathcal{N}.
\end{equation}

\begin{proof}[Proof of Theorem B] 
Let $\mm\in \PM[M]$ be such that its conditionals in $\xi$ are given by $\{\nux\}_{x}$, i.e.
\[
	\mm^{\xi}_x=\nux(\cdot | \xi(x))\quad \mm\aep(x).
\]
In this case, $\mm^{\xi^n}=\nux(\cdot | \xi^n(x))$. Now for $\h\in\CM$ we can compute
\begin{align*}
\int \h d \mm=\int E_n\h(x) d\mm(x)\xrightarrow[n\to\oo]{}\int a d\mm(x)=a
\end{align*}
where we have used \eqref{eq:converEE} and the fact $\mm(M\setminus\mathcal{N})=1$. Since $a$ does not depend on $\mm$ we conclude
\[
	\int \h d\mm=\int \h d\muf\quad \forall \h\in\CM
\]
and therefore, $\mm=\muf$.
\end{proof}

\begin{proof}[Proof of Theorem A]
Let $\tilde{\mu}^{cs}=\{\tilde{\mu}^{cs}_x\}_{x\in M}$ be a quasi-invariant measure for $\Phi^u$ with Jacobian determined by $h\in\Coc$. Construct the corresponding probability measure $\mm$ on $M$ as explained in the last part of the Second section.  By construction $\mm$ has conditionals given by the family $\{\nux\}$, hence by \hyperlink{theoremB}{Theorem B} $\mm=\muf[\Phi,\varphi]$. From here we deduce, using the local product structure of $\mm$ that $\tilde{\mu}^{cs}=\mu^{cs}$. 
\end{proof}

	\newpage
	\printbibliography
\end{document}